\input ./style/arxiv-general.cfg
\documentclass[MSNbibl,number,citesort,dvips]{arxbj}
\makeatletter
   \@ifpackageloaded{graphicx}{}{\usepackage{graphicx}}
\makeatother

\volume{22}
\issue{3}
\pubyear{2016}
\firstpage{1364}
\lastpage{1382}
\doi{10.3150/15-BEJ695}
\docsubty{FLA}

\makeatletter
\newcommand{\ee}{\mathrm{e}}
\newcommand{\p}{{\mathbb P}}
\newcommand{\e}{{\mathbb E}}
\newcommand{\D}{{\mathrm d}}
\newcommand{\1}[1]{1_{\{#1\}}}
\renewcommand{\a}{{\alpha}}
\renewcommand{\b}{{\beta}}
\renewcommand{\l}{{\lambda}}
\renewcommand{\t}{{\theta}}
\newcommand{\vt}{{\vartheta}}
\newtheorem{theorem}{Theorem}[section]
\newtheorem{lemma}{Lemma}[section]
\newremark{remark}{Remark}[section]
\makeatother

\begin{document}
\begin{frontmatter}

\title{Exit identities for L\'evy processes observed at Poisson arrival times}
\runtitle{Exit identities for L\'evy processes observed at Poisson
arrival times}

\begin{aug}
\author[A]{\inits{H.}\fnms{Hansj\"org}~\snm{Albrecher}\corref{}\thanksref{A}\ead[label=e1]{hansjoerg.albrecher@unil.ch}},
\author[B]{\inits{J.}\fnms{Jevgenijs}~\snm{Ivanovs}\thanksref{B}\ead[label=e2]{jevgenijs.ivanovs@unil.ch}}
\and
\author[C]{\inits{X.}\fnms{Xiaowen}~\snm{Zhou}\thanksref{C}\ead[label=e3]{zhou@alcor.concordia.ca}}
\address[A]{Department of Actuarial Science,
Faculty of Business and Economics,
University of Lausanne,
CH-1015 Lausanne,
Switzerland and Swiss Finance Institute, Switzerland.\\ \printead{e1}}
\address[B]{Department of Actuarial Science, Faculty of Business and
Economics, University of Lausanne, CH-1015 Lausanne, Switzerland.
\printead{e2}}
\address[C]{Department of Mathematics and Statistics, Concordia
University, 1455 de Maisonneuve Blvd. W., Montreal, Quebec, Canada H3G 1M8.
\printead{e3}}
\end{aug}

%
\received{\smonth{3} \syear{2014}}
%
\revised{\smonth{9} \syear{2014}}

%
\begin{abstract}
For a spectrally one-sided L\'{e}vy process, we extend various
two-sided exit identities to the situation when the
process is only observed at arrival epochs of an independent Poisson process.
In addition, we consider exit problems of this type for processes
reflected either  from above or from below.
The resulting Laplace transforms of the main quantities of interest are in
terms of scale functions and turn
out to be simple analogues of the classical formulas.
\end{abstract}

%
\end{frontmatter}

\section{Introduction}\label{intro}
Consider a spectrally-negative L\'{e}vy process $X$, that is, a L\'
{e}vy process
with only negative jumps, and which is not a.s. a non-increasing process.
Let
\[
\psi(\t):=\log\e \ee^{\t X(1)},\qquad \t\geq0,
\]
be its Laplace exponent. Denote the law of $X$ with $X(0)=x\ge0$ by
$\p_x$ and the corresponding expectation by $\e_x$.
For a fixed $a\geq0$ define the first passage times
\begin{eqnarray*}
\tau_0^-&:=&\inf\bigl\{t\geq0\dvt X(t)<0\bigr\},\qquad
\tau_a^+:=
\inf\bigl\{t\geq0\dvt X(t)>a\bigr\}.
\end{eqnarray*}
Furthermore, let $T_i$ be the arrival times of an independent Poisson
process of rate $\lambda>0$, and
define the following stopping times:
\begin{eqnarray*}
T_0^-&:=&\min\bigl\{T_i\dvt X(T_i)<0\bigr\},\qquad
T_a^+:=\min\bigl\{T_i\dvt X(T_i)>a\bigr\}.
\end{eqnarray*}
The latter times can be seen as first passage times when $X$ is
observed at Poisson arrival times (by convention $\inf\varnothing=\min
\varnothing=\infty$).

These quantities have useful interpretations in various fields of
applied probability. For instance, if $X$ serves as a model for the
surplus process of an insurance portfolio over time, then $\p_x(\tau
_0^-<\infty)$
is the probability of ruin of the portfolio with initial capital $x$.
Likewise, $\p_x(T_0^-<\infty)$ is the probability that ruin occurs
and is detected, given that the process can only\vadjust{\goodbreak} be monitored at
discrete points in
time modeled by an independent Poisson process.
It is not hard to show that $T_0^-\downarrow\tau_0^-$ a.s. (and
similarly $T_a^+\downarrow\tau_a^+$ a.s.) as the observation rate
$\lambda$ tends to $\infty$,
which may be used to retrieve the classical exit identities.
Quantities related to $T_0^-$ have been studied for a compound Poisson
risk model in \cite{actS}, and in \cite{landriaultoccupation} a
simple formula for $\p_x(T_0^-<\infty)$
was established for general spectrally-negative L\'{e}vy processes $X$.
Ruin-related quantities under surplus-dependent observation rates were
studied in \cite{AL} and for recent results on observation rates
that change according to environmental
conditions, we refer to \cite{albrecherivanovsrisks}.
Poissonian observation is also relevant in queueing contexts, see, for
example, \cite{bekker}.

In practice one may interpret continuous and Poissonian observation as
endogenous and exogenous monitoring of the process of interest, respectively.
For example, in an insurance context $\tau_0^-$ may be understood as
the time of ruin (observed by the insurance company),
whereas $T_a^+$ may be considered as the first time when shareholders
receive dividends (they look at the company at discrete, here random, times).
Hence, the shareholders receive dividends if the event $\{T_a^+<\tau
_0^-\}$ occurs.
Another example comes from reliability theory~\cite{nakagawa2006maintenance},
where one considers a degradation process and assumes that $\tau_0^-$
is the time of failure and $T_a^-:=\min\{T_i\dvt X(T_i)<a\}$ for some
$a>0$ is the time at which the process is
observed in its critical state necessitating replacement. Hence, the
event $\{T_a^-<\tau_0^-\}$ signifies preventive replacement before failure.
Finally, some identities involving both continuous and Poissonian
observation lead to transforms of certain occupation times, see
Remark~\ref{rem}.

In this paper, we establish formulas for two-sided exit probabilities
under Poissonian observation, as well as formulas for
the joint transform of the exit time and the corresponding overshoot on
the event of interest. In addition, we consider reflected processes and
provide the joint transforms
including the total amount of output (dividends) or input (required
capital to remain solvent) up to the exit time.
Note that the formulas also
hold for spectrally-positive L\'{e}vy processes by simply exchanging the
roles of the involved quantities. It turns out that the resulting
formulas have a rather slim form in terms of first and second scale
functions, and
along the way it also proves useful to define a third scale function.
The form of the expressions allows to interpret them as natural
analogues of the respective counterparts under continuous observation.
Finally, we note that discrete observation allows for a wide range of
cases, and our list of exit identities is not exhaustive.
We only consider the basic cases, that is, the ones where the
corresponding events stay non-trivial if Poissonian observation is
replaced by a continuous one, which, for example, excludes $\{
T_a^-<\tau_0^-\}$ mentioned above.

The remainder of this paper is organized as follows. Section~\ref
{sec2} recalls some relevant exit identities under continuous observation.
Section~\ref{sec4} contains the main results, and the proofs are given
in Section~\ref{sec5}.
Finally, Section~\ref{sec7} gives an illustration of some identities
for the case of Cram\'er--Lundberg risk model with exponential claims.

Throughout this work, we use $e_u$ to denote an exponentially
distributed r.v.
with rate $u>0$, which is independent of everything else.


\section{Standard exit theory}\label{sec2}
The most basic identity states that
%
\begin{equation}
\label{bas}\p_0\bigl(\tau_a^+<\infty
\bigr)=\ee^{-\Phi a},\qquad a\geq0,
\end{equation}
where $\Phi\geq0$ is the right-most non-negative solution of $\psi
(\theta)=0$.
Let us recall two fundamental functions which enter various exit identities.
The (first) scale function $W(x)$ is a non-negative function, with
$W(x)=0$ for $x<0$, continuous on $[0,\infty)$, positive for positive $x$,
and characterized by the transform
\[
\int_0^\infty \ee^{-\t x}W(x)\,\D x=1/\psi(\t),\qquad
\t>\Phi.
\]
It enters the basic two-sided exit identity for $a>0$ through
%
\begin{equation}
\p_x\bigl(\tau_a^+<\tau_0^-\bigr)=W(x)/W(a),
\qquad x\leq a,\label{eqW}
\end{equation}
see, for example, \cite{kyprianou}.

The so-called second scale function is defined by
%
\begin{equation}
\label{eqZ} Z(x,\t):=\ee^{\t x} \biggl(1-\psi(\t)\int_0^x
\ee^{-\t y}W(y)\,\D y \biggr),\qquad x\geq0
\end{equation}
and $Z(x,\t):=\ee^{\t x}$ for $x<0$. Note that for $\t=0$, $Z(x,\t)$
reduces to $Z(x)$ as defined in
\cite{kyprianou}, Chapter~2.
It is convenient to define $Z$ as a function of two arguments, which
allows to provide more general formulas.
We refer to~\cite{ivanovsscale} for this definition and the following
formulas in a more general setting of Markov additive processes.
Note that for $\t>\Phi$ we can rewrite $Z(x,\t)$ in the form
%
\begin{eqnarray}
Z(x,\t)=\psi(\t)\int_0^\infty \ee^{-\theta
y}W(x+y)\,\D y,\qquad x\geq0, \t>\Phi.\label{eqZalt}
\end{eqnarray}
It is known that for $x\leq a$ one has
%
\begin{eqnarray}
\e_x \bigl(\ee^{\t X(\tau_0^-)};\tau_0^-<
\tau_a^+ \bigr)=Z(x,\t)-W(x)\frac{Z(a,\t)}{W(a)}.\label{eqexit0}
\end{eqnarray}
Moreover, for $\t>\Phi$ we have
%
\begin{equation}
\label{eqlim} \lim_{a\rightarrow\infty}Z(a,\t)/W(a)=\psi(\t)/(\t-\Phi)
\end{equation}
and so we also find that
%
\begin{equation}
\label{eqexit00}\e_x \bigl(\ee^{\t X(\tau_0^-)},\tau_0^-<\infty
\bigr)=Z(x,\t)-W(x)\frac{\psi(\t)}{\t-\Phi}.
\end{equation}

Importantly, all the above results hold for an exponentially killed
process $X$ (cf.~\cite{ivanovskilling}): 
For a killing rate $q>0$, we write $\psi_q(\t)=\psi(\t)-q$, then
$\Phi_q>0$ is the positive solution to $\psi_q(\t)=0$, and $W_q(x)$
is defined by
\[
\int_0^\infty \ee^{-\t x}W_q(x)\,\D
x=1/\psi_q(\t),\qquad \t>\Phi_q.
\]
With $Z_q(x,\theta)$ defined through $W_q(x)$ and $\psi_q(\theta)$,
formula (\ref{eqexit0}) in the case of killing reads
\begin{eqnarray*}
\e_x \bigl(\ee^{-q \tau_0^-+\t X(\tau_0^-)};\tau_0^-<
\tau_a^+ \bigr) &=& \e_x \bigl(\ee^{\t X(\tau_0^-)};
\tau_0^-<\tau_a^+,\tau_0^-<e_q
\bigr)
\\
&=& Z_q(x,\t)-W_q(x)\frac{Z_q(a,\t)}{W_q(a)},
\end{eqnarray*}
and hence the information on the time of the exit is easily added. The
same adaptations hold for the other exit identities above.
For the sake of readability, we will often drop the index $q$ in the
sequel, if it does not cause confusion.
In this case, $\Phi_\l,W_\l(x),Z_\l(x,\t)$ should be interpreted
as $\Phi_{\l+q}, W_{\l+q}(x),Z_{\l+q}(x,\t)$, respectively, that is,
they correspond to the process killed at rate $q$ and then additionally
killed at rate $\l$.

Finally, we will need the following identities which can readily be
obtained from the known formulas for potential densities of $X$ killed
upon exiting a certain interval,
see, for example, \cite{bertoinexpdecay} or \cite{kyprianou}, Chapter~8.4:
%
\begin{eqnarray}
\p\bigl(X(e_\l)\in\D x\bigr)&=&\l\bigl(\ee^{-\Phi_\l x}/
\psi'(\Phi_\l)-W_\l(-x) \bigr)\,\D
x,\label{eqpotential}
\\
\p\bigl(X(e_\lambda)\in\D x,e_\lambda<\tau_a^+
\bigr)&=&\l\bigl(\ee^{-\Phi_\l
a}W_\l(a-x)-W_\l(-x)
\bigr)\,\D x,\label{eqpotential1}
\\
\p_a\bigl(X(e_\l)\in\D x,e_\l<
\tau_0^-\bigr)&=&\l\bigl(\ee^{-\Phi_\l x}W_\l
(a)-W_\l(a-x) \bigr)\,\D x,\label{eqpotential2}
\end{eqnarray}
where $a>0$ and the killing rate $q\geq0$ is implicit.

\section{Results}\label{sec4}
One of the first general results concerning $T_0^-$ was obtained
in~\cite{landriaultoccupation}, where it was shown for $q=0$ and $\e
X(1)>0$ that
%
\begin{equation}
\p_x\bigl(T_0^-=\infty\bigr)=\psi'(0)
\frac{\Phi_\l}{\l
}Z(x,\Phi_\l),\qquad x\geq0,\label{eqfirst}
\end{equation}
cf.~(\ref{eqZalt}).
This leads to a strikingly simple identity for $x=0$:
$\p(T_0^-=\infty)=\psi'(0){\Phi_\l}/{\l}$. A more general result
was recently obtained in~\cite{albrecherivanovsrisks} for an
arbitrary killing rate $q\geq0$:
%
\begin{equation}
\p_x\bigl(\tau_a^+<T_0^-\bigr)=
\frac{Z(x,\Phi_\l)}{Z(a,\Phi
_\l)},\qquad x\in[0,a].\label{eqrisk}
\end{equation}
It is easy to see that both these results also hold for $x<0$.
Note the resemblance of~(\ref{eqW}) and (\ref{eqrisk}), and, moreover,
%
\begin{equation}
\label{eqlimZZ}\frac{Z(x,\Phi_\l)}{Z(a,\Phi_\l
)}=\p_x\bigl(\tau_a^+<T_0^-
\bigr)\longrightarrow\p_x\bigl(\tau_a^+<\tau
_0^-\bigr)=\frac{W(x)}{W(a)}\qquad\mbox{as }\l\rightarrow\infty,
\end{equation}
because in the limit $\l\rightarrow\infty$ the Poisson observation
results in continuous observation of the process a.s. (although it is
hard to see the convergence of the ratio directly).

We now present various exit identities for continuous and Poisson
observations extending the standard exit theory.

\begin{theorem}\label{thm1}
For $a,\theta\geq0,x\leq a$ and implicit killing rate $q\geq0$, we have
%
\begin{eqnarray}
\e_x \bigl(\ee^{\theta X(T_0^-)};T_0^-<\infty\bigr)&=&
\frac{\lambda
}{\lambda-\psi(\theta)} \biggl(Z(x,\t)-Z(x,\Phi_\l)\frac{\psi(\t
)(\Phi_\l-\Phi)}{\l(\t-\Phi)}
\biggr),\label{eq1}
\\[-2pt]
\e_x \bigl(\ee^{\theta X(T_0^-)};T_0^-<\tau_a^+
\bigr)&=&\frac{\lambda
}{\lambda-\psi(\theta)} \biggl(Z(x,\t)-Z(x,\Phi_\l)
\frac{Z(a,\t
)}{Z(a,\Phi_\l)} \biggr),\label{eq2}
\\[-2pt]
\e_x \bigl(\ee^{-\t(X(T_a^+)-a)};T_a^+<\infty\bigr)&=&
\frac{\Phi_\l
-\Phi}{\Phi_\l+\theta}\ee^{-\Phi(a-x)},\label{eq3}
\\[-2pt]
\e_x \bigl(\ee^{-\t(X(T_a^+)-a)};T_a^+<\tau_0^-
\bigr)&=&\frac{\l
}{\Phi_\l+\t} \frac{W(x)}{Z(a,\Phi_\l)},\label{eq4}
\\[-2pt]
\e_x \bigl(\ee^{\t X(\tau_0^-)};\tau_0^-<T_a^+
\bigr)&=&Z(x,\t)-\frac
{W(x)}{\t-\Phi_\l} \biggl(\psi(\t)-\l\frac{Z(a,\t)}{Z(a,\Phi_\l
)}
\biggr),\label{eq7}
\end{eqnarray}
where ratios for $\theta=\Phi$ and $\theta=\Phi_\lambda$ should be
interpreted in the limiting sense.
\end{theorem}

\begin{remark}
When $\e X(1)>0$ and $q=\t=0$, we have $\Phi=0$ and $Z(x,0)=1$, so
that~(\ref{eq1}) reduces to (\ref{eqfirst}). Secondly, note the
resemblance between (\ref{eqexit0}) and (\ref{eq2}), and that the
first is retrieved from the second when $\l\rightarrow\infty$
cf.~(\ref{eqlimZZ}).
Next, (\ref{eq3}) is an extension of identity (\ref{bas}) (which is
retained for $\theta=0$ and $\l\rightarrow\infty$).
This formula also implies that the overshoot $X(T_a^+)-a$ given $\{
T_a^+<\infty\}$ and $q=0$ is exponentially distributed with rate $\Phi
_\l$ for all $x\leq a$
(indeed it is not hard to establish the memoryless property of this overshoot).
The identity (\ref{eq4}) is a variation of~(\ref{eqW}) and (\ref
{eqrisk}), and identity (\ref{eq7}) is the counterpart of~(\ref
{eq4}) for the process $-X$ (reproducing (\ref{eqexit0}) for $\l
\rightarrow\infty$ and
(\ref{eqexit00}) for $\l\downarrow0$).
\end{remark}

\begin{remark}\label{rem}
There is a close link between some of our results and transforms of
certain occupation times:
\begin{eqnarray*}
\e_x \bigl(\ee^{\t X(\tau_0^-)};\tau_0^-<T_a^+
\bigr)&=&\e_x \bigl(\ee^{\t X(\tau_0^-)};\tau_0^-<\infty, N(A)=0
\bigr)
\\[-2pt]
&=& \e_x \bigl(\ee^{\t X(\tau_0^-)-\l\int_0^{\tau_0^-}1_{\{X(t)>a\}}\,\D
t};\tau_0^-<\infty
\bigr),
\end{eqnarray*}
where $A=\{t\in[0,\tau_0^-)\dvt X(t)>a\}$ and $N$ is an independent
Poisson random measure with intensity $\lambda\,\D t$. A similar
identity also holds for $\p_x(\tau_a^+<T_0^-)$.
Hence,~(\ref{eqrisk}) and~(\ref{eq7}) for $\t=0$ can alternatively
be obtained from~\cite{loeffenocctimes}, Corollaries~1 and 2,
and taking appropriate limits followed by somewhat tedious simplifications.
\end{remark}

The next result considers two-sided exit for exclusively Poissonian
observation of the process.

\begin{theorem}\label{thm2}
For $a,\theta\geq0, x\leq a$ and implicit killing rate $q\geq0$, we have
%
\begin{eqnarray}
\e_x \bigl(\ee^{\theta X(T_0^-)};T_0^-<T_a^+
\bigr)&=&\frac{\l}{\l
-\psi(\t)} \biggl(Z(x,\t)-Z(x,\Phi_\l)
\frac{\tilde Z(a,\Phi_\l,\t)}{\tilde Z(a,\Phi_\l,\Phi_\l)} \biggr),
\label{eq5}
\\[-2pt]
\e_x \bigl(\ee^{-\t(X(T_a^+)-a)};T_a^+<T_0^-
\bigr)&=&\frac{\l}{\Phi
_\l+\theta}\frac{Z(x,\Phi_\l)}{\tilde Z(a,\Phi_\l,\Phi_\l
)},\label{eq6}
\end{eqnarray}
where we define a third scale function as
%
\begin{equation}
\label{eqZZ} \tilde Z(x,\a,\b):=\frac{\psi(\a)Z(x,\b)-\psi(\b)Z(x,\a
)}{\a
-\b},\qquad\a,\b\geq0.
\end{equation}
\end{theorem}

Again there is a striking similarity between (\ref{eq2}) and (\ref
{eq5}), as well as between (\ref{eq4}) and (\ref{eq6}).
Note that for $\a=\b$ the definition~(\ref{eqZZ}) results in
\[
\tilde Z(x,\a,\a)=\psi'(\a)Z(x,\a)-\psi(\a)Z'(x,\a),
\]
where the differentiation of $Z$ is with respect to the second argument.



We now present results for reflected processes.
Write $\e_x^0$ for the law of $X$ reflected at 0 (from below) and $\e
_x^a$ for the law of $X$ reflected at $a$ from above,
and let $R$ be the regulator at the corresponding barrier.
That is $(X(t),R(t))$ under $\e_x^0$ and under $\e_x^a$ is given by
$(X(t)+(-\underline X(t))^+,(-\underline X(t))^+)$ and
$(X(t)-(\overline X(t)-a)^+,(\overline X(t)-a)^+)$ under $\e_x$, respectively,
where
\begin{eqnarray*}
\underline X(t) &:=& \inf\bigl\{X(s)\dvt0\le s \le t\bigr\},\qquad \overline
X(t):=\sup\bigl\{ X(s)\dvt0\le s \le t\bigr\}.
\end{eqnarray*}

\begin{theorem}\label{thmreflected}
For $a>0,\t,\vt\geq0, x\leq a$ and implicit killing rate $q\geq0$,
we have the following identities for the reflected processes:
%
\begin{eqnarray}
&& \e_x^0 \bigl(\ee^{-\vt R(T_a^+)-\t(X(T_a^+)-a)};T_a^+<\infty\bigr)
\nonumber\\[-8pt]\label{refl1} \\[-8pt]\nonumber
&&\quad =\frac{\l(\vt-\Phi_\l)Z(x,\vt)}{(\Phi_\l+\t)(\psi(\vt
)Z(a,\Phi_\l)-\l Z(a,\vt))}=\frac{\l}{\Phi_\l+\t}\frac{Z(x,\vt
)}{\tilde Z(a,\vt,\Phi_\l)},
\\
&& \e_x^a \bigl(\ee^{-\vt R(T_0^-)+\t X(T_0^-)};T_0^-<\infty
\bigr)
\nonumber\\[-8pt]\label{refl2} \\[-8pt]\nonumber
&&\quad = \frac{\l}{\l-\psi(\t)} \biggl(Z(x,\t)+Z(x,\Phi_\l)
\frac{W(a)\psi(\t)-(\t+\vt)Z(a,\t)}{Z'(a,\Phi_{\lambda})+\vartheta Z(a,\Phi_{\lambda})} \biggr),
\end{eqnarray}
where the derivative of $Z$ is taken with respect to the first argument.
\end{theorem}

Note that $T_a^+$ and $T_0^-$ can be infinite due to the implicit
killing rate $q$. This result for $\l=\infty$ is to be compared with
%
\begin{eqnarray}
&& \e_x^0 \bigl(\ee^{-\vt R(\tau_a^+)};\tau_a^+<\infty\bigr) = \frac
{Z(x,\vt)}{Z(a,\vt)},\label{reflectedclassical1}
\\
&& \e^a_x \bigl(\ee^{-\vt R(\tau_0^-)+\t X(\tau_0^-)};\tau_0^-<
\infty\bigr)
\nonumber\\[-8pt]\label{reflectedclassical2} \\[-8pt]\nonumber
&&\quad = Z(x,\t)+W(x)\frac{W(a)\psi(\t)-(\t+\vt)Z(a,\t)}{W'_+(a)+\vt W(a)},
\end{eqnarray}
where $W'_+$ denotes the right derivative of $W$ (see, e.g.,~\cite
{kyprianou}, Theorem~8.10, for $\vt=0$ and~\cite{ivanovsscale} for the
general case).
Furthermore, letting $a\rightarrow\infty$ in (\ref{refl2}) and
using~(\ref{eqlim}) we obtain~(\ref{eq1}), which provides a nice
check ($\vt$ indeed cancels out).
Similarly, (\ref{refl1}) leads to~(\ref{eq3}) if we put $a=\Delta+x$
and let $x\rightarrow\infty$ (note that~(\ref{eq3}) depends only on
the difference~$\Delta$).

Finally, we note that yet another exit identity for a reflected process
with Poissonian observations can be found in~\cite{poweridentities},
Corollary~6.1.
In particular, letting $\rho_y:=\inf\{t\geq0\dvt R(t)>y\}$ be the first
passage time of $R$, it holds that
%
\begin{equation}
\label{idthi} \p_x^a\bigl(\rho_y<T_0^-
\bigr)=\frac{Z(x,\Phi_\l)}{Z(a,\Phi_{\l})}\exp\biggl(-\frac{Z'(a,\Phi
_{\l})}{Z(a,\Phi_{\l})}y \biggr),
\end{equation}
where $q\geq0$ is implicit and $x\in[0,a]$. 
If $X$ is some surplus process, it is natural to interpret $R(t)$ as
the dividend payments up to time $t$ according to a horizontal dividend
barrier strategy.
Identity~(\ref{idthi}) then allows to obtain the expected discounted
dividends until ruin:
%
\begin{eqnarray}\label{eqdividends1}
\e_x^a\int_0^\infty
\ee^{-q t}\1{t<T_0^-}\,\D R(t) &=&\e_x^a\int
_0^\infty \ee^{-q \rho_y}\1{\rho_y<T_0^-}\,\D y\nonumber
\\
&=&\int_0^\infty\e_x^a
\bigl(\ee^{-q \rho_y};\rho_y<T_0^-\bigr)\,\D y
\nonumber\\[-8pt]\\[-8pt]\nonumber
&=&\frac{Z_q(x,\Phi_{\l+q})}{Z_q(a,\Phi_{\l+q})}\int_0^\infty\exp\biggl(-
\frac{Z_q'(a,\Phi_{\l+q})}{Z_q(a,\Phi_{\l+q})}y \biggr)\,\D y
\\
&=& \frac{Z_q(x,\Phi_{\l+q})}{Z_q'(a,\Phi_{\l+q})}, \nonumber
\end{eqnarray}
where $x\in[0,a]$.
In the case of continuous observations, that is, $\l=\infty$, this
expression reduces to $W_q(x)/W_{q+}'(a)$, see, for example,~\cite
{dividends}, Proposition~2. Also, in the absence of discounting ($q=0$), one
obtains from (\ref{refl2})
for $\t=0$ that
\[
\e_a^a \bigl(\ee^{-\vt R(T_0^-)} \bigr)=Z'(a,\Phi_{\lambda})/\bigl(Z'(a,\Phi_{\lambda})+\vartheta Z(a,\Phi_{\lambda})\bigr),
\]
that is, the distribution of total dividend payments until ruin
(observed at Poissonian times) is exponentially distributed with
parameter $Z'(a,\Phi_{\lambda})/Z(a,\Phi_{\lambda})$, if the initial
surplus level is at the barrier.
The exponential parameter reduces to $W_+'(a)/W(a)$ for $\l\to\infty
$, cf. \cite{dividends}, Section~5.


\begin{remark}
We emphasize again that each of the formulas in Theorems \ref{thm1}--\ref{thmreflected} can also be written for $q> 0$ in an
explicit form, see also~(\ref{eqdividends1}). For instance, (\ref{eq2})
can be read as
\[
\e_x \bigl(\ee^{-qT_0^-+\theta X(T_0^-)};T_0^-<\tau_a^+
\bigr)=\frac
{\lambda}{\lambda-\psi_q(\theta)} \biggl(Z_q(x,\t)-Z_q(x,
\Phi_{\l
+q} )\frac{Z_q(a,\t)}{Z_q(a,\Phi_{\l+q} )} \biggr).
\]
\end{remark}

\section{Proofs}\label{sec5}

Some of the proofs below will rely on the following intriguing identity
first observed in~\cite{loeffenocctimes}, Equation~(6):
%
\begin{equation}
\label{eqscaleconv} (p-q)\int_0^a
W_p(a-x)W_q(x)\,\D x=W_p(a)-W_q(a),
\end{equation}
which as a consequence yields
\[
\int_0^a W_q(a-x)W_q(x)\,
\D x=\frac{\partial W_q(a)}{\partial q}.
\]
The following result generalizes the second part of \cite
{loeffenocctimes}, Equation~(6):

\begin{lemma}\label{lemZrepr}
For $\t,\a,p,q\geq0$, it holds that
\[
(p-q)\int_0^aW_p(a-x)Z_q(x,
\t)\,\D x=Z_p(a,\t)-Z_q(a,\t).
\]
\end{lemma}

\begin{pf}
First, we show that
%
\begin{eqnarray}\label{eqZrept1}
&& (p-q)\int_0^aW_p(a-x)
\int_0^x \ee^{\t(x-y)}W_q(y)\,\D y\,
\D x
\nonumber\\[-8pt]\\[-8pt]\nonumber
&&\quad =\ee^{\t a} \biggl(\int_0^a
\ee^{-\t x}W_p(x)\,\D x-\int_0^a
\ee^{-\t x}W_q(x)\,\D x \biggr)
\end{eqnarray}
by taking transforms of both sides.
The left-hand side gives, for sufficiently large $s$,
\[
\int_0^\infty \ee^{-s a} \biggl((p-q)\int
_0^aW_p(a-x)\int
_0^x \ee^{\t
(x-y)}W_q(y)\,\D y\,\D x
\biggr)\,\D a=\frac{p-q}{(s-\t)\psi_p(s)\psi_q(s)},
\]
and for the right-hand side we have
\[
\int_0^\infty \ee^{-s a}
\biggl(\ee^{\t a}\int_0^a
\ee^{-\t x}W_p(x)\,\D x \biggr)\,\D a=\frac{1}{(s-\t)\psi_p(s)}
\]
and similarly for the second term. Then~(\ref{eqZrept1}) follows by
noting that
\[
\frac{1}{\psi_p(s)}-\frac{1}{\psi_q(s)}=\frac{p-q}{\psi_p(s)\psi_q(s)}.
\]
Finally, using~(\ref{eqZrept1}) we get
\begin{eqnarray*}
&&(p-q)\int_0^aW_p(a-x)Z_q(x,
\t)\,\D x
\\
&&\quad =(p-q)\int_0^a W_p(a-x)
\ee^{\t
x}\,\D x-\psi_q(\t)\ee^{\t a} \biggl(\int
_0^a \ee^{-\t x}W_p(x)\,\D x-
\int_0^a \ee^{-\t x}W_q(x)\,\D x
\biggr)
\\
&&\quad = \ee^{\t a}\bigl(p-q-\psi_q(\t)\bigr)\int
_0^a \ee^{-\t x}W_p(x)\,\D x+\ee^{\t a}\psi_q(\t)\int_0^a
\ee^{-\t x}W_q(x)\,\D x
\\
&&\quad =Z_p(a,\t)-Z_q(a,
\t)
\end{eqnarray*}
finishing the proof.
\end{pf}

\subsection{Proof of Theorem~\texorpdfstring{\protect\ref{thm1}}{3.1}}
We split the proof into several parts.

\begin{pf*}{Proof of Equation~(\ref{eq2})}
Denoting $f(x,\t,a):=\e_x (\ee^{\theta X(T_0^-)};T_0^-<\tau
_a^+ )$ one can write, using the strong Markov property,
\begin{eqnarray*}
f(x,\t,a) &=& \int_{-\infty}^0\p_x\bigl(X
\bigl(\tau_0^-\bigr)\in\D z,\tau_0^-<\tau_a^+
\bigr) \bigl(\p_z\bigl(\tau_0^+<e_\l
\bigr)f(0;\t,a)+\e_z\bigl(\ee^{\theta X(e_\lambda
)},e_\lambda<
\tau_0^+\bigr) \bigr).
\end{eqnarray*}
Recall that $\p_z(\tau_0^+<e_\lambda)=\ee^{\Phi_\l z},z\leq0$ and also
%
\begin{eqnarray}
\label{eqexptime} \e_z\bigl(\ee^{\theta X(e_\lambda)},e_\lambda<
\tau_0^+\bigr)=\e_z \bigl(\ee^{\theta X(e_\lambda)}
\bigr)-\ee^{\Phi_\l z}\e\bigl(\ee^{\theta
X(e_\lambda)} \bigr)=\frac{\lambda}{\lambda-\psi(\theta
)}
\bigl(\ee^{\theta z}-\ee^{\Phi_\l z}\bigr)
\end{eqnarray}
for $\theta$ small enough such that $\psi(\theta)<\lambda$. The
result can then be analytically continued to any $\t\geq0$.
Thus, using~(\ref{eqexit0}) we arrive at
%
\begin{eqnarray}\label{eqtowork}
f(x,\t,a) &=& \biggl(Z(x,\Phi_\l)-Z(a,
\Phi_\l)\frac{W(x)}{W(a)} \biggr) \biggl(f(0,\t,a)-\frac{\lambda
}{\lambda-\psi(\theta)}
\biggr)
\nonumber\\[-8pt]\\[-8pt]\nonumber
&&{} + \biggl(Z(x,\t)-Z(a,\t)\frac{W(x)}{W(a)} \biggr)\frac{\lambda
}{\lambda-\psi(\theta)}.
\end{eqnarray}

Note that due to $Z(0,\t)=1$ we get for $x=0$
\[
Z(a,\Phi_\l)\frac{W(0)}{W(a)}f(0,\t,a)=\frac{\lambda}{\lambda
-\psi(\theta)}
\biggl(Z(a,\Phi_\l)\frac{W(0)}{W(a)}-Z(a,\t)\frac
{W(0)}{W(a)}
\biggr).
\]
This equation is trivial when $W(0)=0$, that is, when $X$ has sample
paths of unbounded variation, see, for example,~\cite{kyprianou},
Equation~(8.26), but otherwise we have
%
\begin{equation}
\label{eqf0} f(0,\t,a)=\frac{\lambda}{\lambda-\psi(\theta)} \biggl
(1-\frac
{Z(a,\t)}{Z(a,\Phi_\l)} \biggr)
\end{equation}
and the result follows combining (\ref{eqtowork}) and (\ref{eqf0}).

It is only left to show that~(\ref{eqf0}) holds also when $X$ has
sample paths of unbounded variation, that is, when $W(0)=0$.
For $x\in(0,a)$, we have
%
\begin{eqnarray}\label{eqf0new}
&& f(0,\t,a)=\p\bigl(\tau_x^+<e_\l\bigr)f(x,
\t,a)+A(x)+B(x), \nonumber
\\
&&\quad\mbox{where } A(x):=\e\bigl(\ee^{\t X(e_\lambda)};e_\lambda<\tau_x^+,X(e_\lambda
)<0\bigr),
\\
&&\quad B(x) := \int_0^x \p\bigl(e_\lambda<
\tau_x^+,X(e_\lambda)\in\D y\bigr)f(y,\t,a).
\nonumber
\end{eqnarray}

It is well known that for $\theta\geq0$
%
\begin{equation}
\label{eqlimW} \int_0^x \ee^{-\t y}W(y)\,\D
y/W(x)\rightarrow0\qquad\mbox{as }x\downarrow0,
\end{equation}
which can be seen by interpreting the ratio of scale functions. In a
similar way one can show, using (\ref{eqscaleconv}), that $W_\l
(x)/W(x)\rightarrow1$ as $x\downarrow0$.
Using (\ref{eqpotential1}), observe that
\[
\p\bigl(e_\lambda<\tau_x^+,X(e_\lambda)\geq0\bigr)=
\l \ee^{-\Phi_\l x}\int_0^x W_\l(y)\,\D y=\mathrm{o}\bigl(W(x)\bigr)
\]
as $x\downarrow0$.
Next, using~(\ref{eqexptime}) observe that $B(x)=\mathrm{o}(W(x))$ and
\begin{eqnarray*}
A(x)&:=& \e\bigl(\ee^{\theta X(e_\lambda)}; e_\lambda<\tau^+_x\bigr) -
\e\bigl(\ee^{\theta X(e_\lambda)}; e_\lambda<\tau^+_x,
X(e_\lambda)\geq0\bigr)
\\
&=&\frac{\lambda}{\lambda-\psi(\theta)}
\bigl(1-\ee^{(\theta-\Phi
_{\lambda})x} \bigr)+\mathrm{o}\bigl(W(x)\bigr).
\end{eqnarray*}

Plugging (\ref{eqf0new}) into (\ref{eqtowork}) and rearranging it
we obtain
%
\begin{eqnarray}\label{eqtotakelim}
&& f(x,\theta,a) \biggl[ 1- \biggl(Z(x,\Phi_{\lambda})-Z(a,
\Phi_{\lambda
})\frac{W(x)}{W(a)} \biggr)\ee^{-\Phi_{\lambda}x} \biggr]\nonumber
\\
&&\quad =-\frac{\lambda}{\lambda-\psi(\theta)}
\nonumber\\[-8pt]\\[-8pt]\nonumber
&&\qquad\hspace*{6pt} {}\times  \biggl(Z(x,\Phi_{\lambda
})\ee^{(\t-\Phi_{\lambda})x} -Z(x,
\theta)+\frac{W(x)}{W(a)}\bigl(Z(a,\theta)-Z(a,\Phi_{\lambda
})\ee^{(\t-\Phi_{\lambda})x}
\bigr) \biggr)
\\
&&\qquad{} +\mathrm{o}\bigl(W(x)\bigr).\nonumber 
\end{eqnarray}

Divide (\ref{eqtotakelim}) by $W(x)$ and take the limit as
$x\downarrow0$, using the representation~(\ref{eqZ}) and then
also~(\ref{eqlimW}), to obtain
%
\[
f(0,\theta,a)Z(a,\Phi_{\lambda})\frac{1}{W(a)}=-\frac{\lambda
}{\lambda-\psi(\theta)}
\frac{1}{W(a)}\bigl(Z(a,\theta)-Z(a,\Phi_{\lambda})\bigr),
\]
which immediately yields~(\ref{eqf0}).
\end{pf*}

\begin{pf*}{Proof of Equation~(\ref{eq4})}
We only need to consider $x\in[0,a]$.
Putting
\[
f(x,\t):=\e_x\bigl(\ee^{-\t(X(T_a^+)-a)};T_a^+<\tau_0^-\bigr)
\]
we write
%
\begin{equation}
\label{eqeq4proof}f(x,\t)=\p_x\bigl(\tau_a^+<\tau
_0^-\bigr)f(a,\t)=\frac{W(x)}{W(a)}f(a,\t).
\end{equation}
%
Using~(\ref{eqpotential2}) and conditioning on the first Poisson
observation time, we get
\begin{eqnarray*}
f(a,\t)=\int_a^\infty \ee^{-\t(x-a)}\l
W_\l(a)\ee^{-\Phi_\l x}\,\D x+\int_0^a
\l\bigl(W_\l(a)\ee^{-\Phi_\l x}-W_\l(a-x)\bigr)f(x,\t)\,\D
x.
\end{eqnarray*}
Using~(\ref{eqeq4proof}), we obtain
\begin{eqnarray*}
&& f(a,\t) \biggl(W(a)-\l W_\l(a)\int_{0}^a
W(x)\ee^{-\Phi_\l x}\,\D x+\l\int_{0}^a
W_\l(x)W(a-x)\,\D x \biggr)
\\
&&\quad =\frac{\l W_\l(a)W(a)}{\Phi_\l
+\t}\ee^{-\Phi_\l a}.
\end{eqnarray*}
With the help of~(\ref{eqscaleconv}), the expression in the brackets
reduces to
\[
W_\l(a) \biggl(1-\l\int_{0}^a
W(x)\ee^{-\Phi_\l x}\,\D x \biggr)=W_\l(a)\ee^{-\Phi_\l a}Z(a,
\Phi_\l),
\]
which shows that $f(a,\t)=\frac{\l}{\Phi_\l+\t}\frac
{W(a)}{Z(a,\Phi_\l)}$, completing the proof in view of~(\ref{eqeq4proof}).
\end{pf*}

\begin{pf*}{Proof of Equations~(\ref{eq1}) and (\ref{eq3})}
Identity~(\ref{eq1}) for $\t>\Phi$ follows immediately from (\ref
{eq2}) and~(\ref{eqlim}); by analytic continuation it is also true
for any $\t\geq0$.
Similarly, (\ref{eq3}) follows from (\ref{eq4}) by plugging in $x+u$
and $a+u$ instead of $x$ and $u$, respectively, letting $u\rightarrow
\infty$
and using~(\ref{eqlim}) together with
\[
\lim_{u\rightarrow\infty}W(x+u)/W(a+u)=\p_x\bigl(
\tau_a^+<\infty\bigr)=\ee^{-\Phi(a-x)}.
\]\upqed
\end{pf*}

\begin{pf*}{Proof of Equation~(\ref{eq7})}
Consider $f(x):=\e_x(\ee^{\t X(\tau_0^-)};\tau_0^-<T_a^+)$, which can
be written as
%
\begin{eqnarray}\label{eqproof7}
f(x)&=& \p_x\bigl(\tau_a^+<
\tau_0^-\bigr)f(a)+\e_x\bigl(\ee^{\t X(\tau_0^-)};\tau
_0^-<\tau_a^+\bigr)
\nonumber\\[-8pt]\\[-8pt]\nonumber
&=& \frac{W(x)}{W(a)}f(a)+Z(x,\t)-W(x)
\frac{Z(a,\t)}{W(a)}
\end{eqnarray}
and also
\[
f(a)=\int_{0}^a\p_a
\bigl(X(e_\l)\in\D x,e_\l<\tau_0^-
\bigr)f(x)+\e_a\bigl(\ee^{\t
X(\tau_0^-)};\tau_0^-<e_\l
\bigr),
\]
where the last term is $Z_\l(a,\t)-W_\l(a)\frac{\psi(\t)-\l}{\t
-\Phi_\l}$ according to~(\ref{eqexit00}).
Hence, we can determine $f(a)$ by plugging in~(\ref{eqproof7}) and
computing the integrals of $W(x)$ and $Z(x,\t)$ with respect to~(\ref
{eqpotential2}).
In particular, using~(\ref{eqscaleconv}) we find
\begin{eqnarray*}
&& \int_{0}^a\p_a
\bigl(X(e_\l)\in\D x,e_\l<\tau_0^-
\bigr)W(x)
\\
&&\quad = \int_0^a\l\bigl(W_\l(a)\ee^{-\Phi_\l x}-W_\l(a-x)
\bigr)W(x)\,\D x
\\
&&\quad =\l W_\l(a)\int_0^a
\ee^{-\Phi_\l x}W(x)\,\D x-\bigl(W_\l(a)-W(a)\bigr)
\\
&&\quad = W(a)-Z(a,
\Phi_\l)\ee^{-\Phi_\l a}W_\l(a).
\end{eqnarray*}
Next we compute
\begin{eqnarray*}
&& \int_0^a\mathrm{e}^{-\Phi_\l x}Z(x,\t)\,\D x
\\
&&\quad =
\frac{1}{\t-\Phi_\l} \bigl(\ee^{(\t
-\Phi_\l)a}-1\bigr)-\frac{\psi(\t)}{\t-\Phi_\l}
\biggl(\ee^{(\t-\Phi
_\l)a}\int_0^a\mathrm{e}^{-\t y}W(y)\,
\D y-\int_0^a\mathrm{e}^{-\Phi_\l y}W(y)\,\D y \biggr)
\\
&&\quad =\frac{1}{\t-\Phi_\l} \biggl(\ee^{-\Phi_\l a}Z(a,\t)-1+\psi(\t)\int
_0^a\mathrm{e}^{-\Phi_\l y}W(y)\,\D y \biggr),
\end{eqnarray*}
which together with Lemma~\ref{lemZrepr} implies that
\begin{eqnarray*}
T &:=& \int_{0}^a\p_a
\bigl(X(e_\l)\in\D y,e_\l<\tau_0^-\bigr)Z(y,
\t)
\\
&=&\frac{\l W_\l(a)}{\t-\Phi_\l} \biggl(\ee^{-\Phi_\l a}Z(a,\t)-1+\psi(\t
)\int
_0^a\mathrm{e}^{-\Phi_\l y}W(y)\,\D y
\biggr)-Z_\l(a,\t)+Z(a,\t).
\end{eqnarray*}
This finally yields
\begin{eqnarray*}
f(a)&=& \biggl(1-Z(a,\Phi_\l)\ee^{-\Phi_\l a}\frac{W_\l
(a)}{W(a)}
\biggr)f(a)+T- \biggl(1-Z(a,\Phi_\l)\ee^{-\Phi_\l a}
\frac
{W_\l(a)}{W(a)} \biggr)Z(a,\t)
\\
&&{}+Z_\l(a,\t)-W_\l(a)\frac{\psi(\t)-\l}{\t-\Phi_\l},
\end{eqnarray*}
which reduces to
%
\begin{eqnarray*}
&& \biggl(Z(a,\Phi_\l)\ee^{-\Phi_\l a}\frac{W_\l(a)}{W(a)} \biggr)f(a)
\\
&&\quad =\frac{W_\l(a)}{\t-\Phi_\l} \bigl(\l \ee^{-\Phi_\l a}Z(a,\t)-\psi(
\t)\ee^{-\Phi_\l a}Z(a,\Phi_\l) \bigr)+Z(a,\Phi_\l)\ee^{-\Phi_\l
a}
\frac{W_\l(a)}{W(a)}Z(a,\t),
\end{eqnarray*}
and hence
\[
f(a)=Z(a,\t)-\frac{W(a)}{\t-\Phi_\l} \biggl(\psi(\t)-\l\frac
{Z(a,\t)}{Z(a,\Phi_\l)} \biggr).
\]
Now the result follows from~(\ref{eqproof7}).
\end{pf*}

\subsection{Proof of Theorem \texorpdfstring{\protect\ref{thm2}}{3.2}}\label{sec6}
Using~(\ref{eqZ}) and changing the order of integration, we can show that
\[
(\a-\beta)\int_0^a\mathrm{e}^{-\a x}Z(x,\b)\,\D
x=1+\bigl(\ee^{-\a a}Z(a,\a)-1\bigr)\psi(\beta)/\psi(\a)-\ee^{-\a a}Z(a,
\beta),
\]
and hence $\tilde Z$ has an alternative representation
\[
\tilde Z(a,\a,\b)=\ee^{\a a}\frac{\psi(\a)-\psi(\b)}{\a-\b}-\psi(\a)\int
_0^a\mathrm{e}^{\a(a-x)}Z(x,\b)\,\D x.
\]
Plugging in $\a=\Phi_\l$ and $\b=\t$, we obtain
%
\begin{equation}
\label{eqZZalt2}\l\int_0^a\mathrm{e}^{-\Phi_\l x}Z(x,\t
)\,\D x=\frac{\psi(\t)-\l}{\t-\Phi_\l}-\ee^{-\Phi_\l a}\tilde Z(a,\Phi_\l,\t).
\end{equation}

\begin{pf*}{Proof of Equation~(\ref{eq5})}
Defining $f(x,\t):=\e_x (\ee^{\theta X(T_0^-)};T_0^-<T_a^+
)$ for $x\leq a$, we write
\[
f(x,\t)=\e_x \bigl(\ee^{\theta X(T_0^-)};T_0^-<
\tau_a^+ \bigr)+\p_x\bigl(\tau_a^+<T_0^-
\bigr)f(a,\t).
\]
Plugging in the corresponding identities, we first get for $x=0$ that
\[
f(0,\t)=\frac{\lambda}{\lambda-\psi(\theta)} \biggl(1-\frac
{Z(a,\t)}{Z(a,\Phi_\l)} \biggr)+
\frac{f(a,\t)}{Z(a,\Phi_\l)},
\]
and some simplifications yield
%
\begin{equation}
\label{eqproof1} f(x,\t)=\frac{\lambda}{\lambda-\psi(\theta)}\bigl
(Z(x,\t)-Z(x,\Phi_\l)
\bigr)+f(0,\t)Z(x,\Phi_\l).
\end{equation}

Using~(\ref{eqpotential}) and conditioning on the first observation
epoch, we get
\[
f(0,\t)=\frac{\l}{\psi'(\Phi_\l)}\int_0^a\mathrm{e}^{-\Phi_\l x}f(x,
\t)\,\D x+\int_{-\infty}^0\ee^{\t x}\p
\bigl(X(e_\l)\in\D x\bigr),
\]
where the latter term evaluates to $\frac{\l}{\l-\psi(\t)}+\frac
{\l}{\psi'(\Phi_\l)(\t-\Phi_\l)}$.
Plugging in (\ref{eqproof1}), we get
\begin{eqnarray*}
&&f(0,\t) \biggl(1-\frac{\l}{\psi'(\Phi_\l)}\int_0^a\mathrm{e}^{-\Phi_\l
x}Z(x,
\Phi_\l)\,\D x \biggr)
\\
&&\quad =\frac{\l}{\psi'(\Phi_\l)}\frac{\l}{\l-\psi(\t)}\int_0^a\mathrm{e}^{-\Phi_\l x}
\bigl(Z(x,\t)-Z(x,\Phi_\l)\bigr)\,\D x+\frac{\l}{\l-\psi
(\t)}+
\frac{\l}{\psi'(\Phi_\l)(\t-\Phi_\l)}.
\end{eqnarray*}
Using (\ref{eqZZalt2}) this reduces to
\begin{eqnarray*}
&& f(0,\t)\ee^{-\Phi_\l a}\tilde Z(a,\Phi_\l,\Phi_\l)/
\psi'(\Phi_\l)
\\
&&\quad =\ee^{-\Phi_\l a}\frac{\l}{\psi'(\Phi_\l)(\l-\psi(\t
))}\bigl(
\tilde Z(a,\Phi_\l,\Phi_\l)-\tilde Z(a,
\Phi_\l,\t)\bigr),
\end{eqnarray*}
which readily leads to
\[
f(0,\t)=\frac{\l}{\l-\psi(\t)} \biggl(1-\frac{\tilde Z(a,\Phi_\l,\t
)}{\tilde Z(a,\Phi_\l,\Phi_\l)} \biggr)
\]
and then the result follows from~(\ref{eqproof1}).\vadjust{\goodbreak}
\end{pf*}

\begin{pf*}{Proof of Equation~(\ref{eq6})}
We write for $x\leq a$ that
\begin{eqnarray*}
\e_x \bigl(\ee^{-\t(X(T_a^+)-a)} \bigr)&=& \e_x
\bigl(\ee^{-\t
(X(T_a^+)-a)};T_a^+<T_0^-\bigr)
\\
&&{}+\int
_{-\infty}^0\p_x\bigl(X
\bigl(T_0^-\bigr)\in\D y,T_0^-<T_a^+\bigr)
\e_y \bigl(\ee^{-\t(X(T_a^+)-a)} \bigr).
\end{eqnarray*}
Using~(\ref{eq3}), we obtain
\[
\e_x\bigl(\ee^{-\t(X(T_a^+)-a)};T_a^+<T_0^-
\bigr)=\frac{\Phi_\l-\Phi}{\Phi
_\l+\theta}\ee^{-\Phi a}\bigl(\ee^{\Phi x}-\e_x
\bigl(\ee^{\Phi X(T_0^-)};T_0^-<T_a^+\bigr)\bigr).
\]
With (\ref{eqZZ}), we see that $\tilde Z(a,\Phi_\l,\Phi)=\frac{\l
}{\Phi_\l-\Phi}\ee^{\Phi a}$ and then it follows from~(\ref{eq5}) that
\[
\e_x\bigl(\ee^{\Phi X(T_0^-)};T_0^-<T_a^+
\bigr)=\ee^{\Phi x}-Z(x,\Phi_\l)\frac{\l
}{\Phi_\l-\Phi}\ee^{\Phi a}/
\tilde Z(a,\Phi_\l,\Phi_\l),
\]
which completes the proof.
\end{pf*}

\subsection{Proof of Theorem \texorpdfstring{\protect\ref{thmreflected}}{3.3}}\label
{secreflected}
\begin{pf*}{Proof of Equation~(\ref{refl1})}
Let $f(x):=\e_x^0(\ee^{-\vt R(T_a^+)-\t(X(T_a^+)-a)};T_a^+<\infty)$, then
\[
f(x)=\e_x^0\bigl(\ee^{-\vt R(\tau_a^+)};\tau_a^+<
\infty\bigr)f(a)=\frac
{Z(x,\vt)}{Z(a,\vt)}f(a)
\]
according to~(\ref{reflectedclassical1}), and also
\[
f(a)=\e_a\bigl(\ee^{-\t(X(T_a^+)-a)};T_a^+<
\tau_0^-\bigr)+\e_a\bigl(\ee^{\vt X(\tau
_0^-)};
\tau_0^-<T_a^+\bigr)f(0),
\]
which is equal to $Z(a,\vt)f(0)$.
Plugging in~(\ref{eq4}) and~(\ref{eq7})
we solve for $f(0)$:
\[
f(0)=\frac{\l(\vt-\Phi_\l)}{(\Phi_\l+\t)(\psi(\vt)Z(a,\Phi
_\l)-\l Z(a,\vt))}
\]
and the result follows.
\end{pf*}

\begin{pf*}{Proof of Equation~(\ref{refl2})}
Let $f(x)=\e_x^a(\ee^{-\vt R(T_0^-)+\t X(T_0^-)};T_0^-<\infty)$, then
%
\begin{equation}
\label{refl2start} f(x)=\e_x\bigl(\ee^{\t X(T_0^-)};T_0^-<
\tau_a^+\bigr)+\p_x\bigl(\tau_a^+<T_0^-
\bigr)f(a),
\end{equation}
which using~(\ref{eq2}) and~(\ref{eqrisk}) yields
%
\begin{equation}
\label{eqreflected0} f(0)=\frac{\l}{\l-\psi(\t)} \biggl(1-\frac{Z(a,\t
)}{Z(a,\Phi_\l
)} \biggr)+
\frac{f(a)}{Z(a,\Phi_\l)}.
\end{equation}
Also
\begin{eqnarray*}
f(a)=\int_{-\infty}^0 \bigl(\e^a_a
\bigl(\ee^{-\vt R(\tau_0^-)};X\bigl(\tau_0^-\bigr)\in\D z,
\tau_0^-<\infty\bigr) \bigl(\ee^{\Phi_\l z}f(0)+\e_z
\bigl(\ee^{\t X(e_\l
)};e_\l<\tau_0^+\bigr)\bigr) \bigr).
\end{eqnarray*}
Using (\ref{eqexptime}), we arrive at
%
\begin{eqnarray}\label{proofreflected}
f(a) &=&\e^a_a\bigl(\ee^{-\vt R(\tau_0^-)+\Phi_\l
X(\tau_0^-)}
\bigr)f(0)
\nonumber\\[-8pt]\\[-8pt]\nonumber
&&{} +\frac{\l}{\l-\psi(\t)} \bigl(\e^a_a
\bigl(\ee^{-\vt R(\tau_0^-)
 +\t X(\tau
_0^-)}\bigr)-\e^a_a\bigl(\ee^{-\vt R(\tau_0^-)+\Phi_\l X(\tau_0^-)}
\bigr) \bigr).
\end{eqnarray}
Substituting~(\ref{reflectedclassical2}) and (\ref{eqreflected0})
into (\ref{proofreflected}), we obtain after some simplifications
\begin{eqnarray*}
&& f(a) \biggl((\Phi_\l+\vt)-\l\frac{W(a)}{Z(a,\Phi_\l)} \biggr)
\\
&&\quad =
\frac{\l}{\l-\psi(\t)} \biggl(W(a) \biggl(\psi(\t)-\l\frac{Z(a,\t
)}{Z(a,\Phi_\l)}\biggr)+(
\Phi_\l-\t)Z(a,\t) \biggr),
\end{eqnarray*}
which yields the result after plugging $f(a)$ into~(\ref{refl2start})
and yet another round of simplifications.
\end{pf*}

The above proofs mostly rely on the strong Markov property and various
identities from fluctuation theory. We note that often there are
several possibilities to approach a problem,
but some of them may require significantly more effort to obtain a
simple formula resembling the classical case.
One could, for instance, consider using exit theory of random walks for
a purely Poissonian observation.
This approach builds upon some general formulas, see, for
example,~\cite{doneyovershoot}, Theorem~4, ignoring the crucial
assumption of one-sided jumps.
Consequently, one loses structure, making it hard to rewrite these
formulas in terms of scale functions.
Martingale techniques, as used, for example, in~\cite{avramexitlevy}
to obtain some classical exit identities, also do not seem to be
immediately appropriate for our setting.
Moreover, one needs to guess the right martingale and for this one
typically needs to know already the resulting expression.
Finally, independent exponential inter-observation times may suggest
using Wiener--Hopf factorization, exploited, for example, in~\cite
{kuznetsovMClevy} to design simulation algorithms.
Indeed, this factorization is one way to prove~(\ref
{eqpotential})--(\ref{eqpotential2}), which are building blocks of
our results.

\section{Cram\'er--Lundberg risk model with exponential claims}\label{sec7}
As mentioned in Section~\ref{intro}, one application area for
identities of the above type is the ruin analysis for an insurance
portfolio with surplus value $X(t)=x+ct-S(t)$ at time $t$,
where $x\geq0$ is the initial capital and $c>0$ is a constant premium
intensity. The classical Cram\'er--Lundberg risk model in this context
assumes $S(t)$ to be a compound Poisson process,
where independent and identically distributed claims arrive according
to a homogeneous Poisson process with rate $\nu$ (see, e.g., \cite
{AA}). Assume now that claims are $\operatorname{Exp}(\eta$)
distributed. Then
\[
\psi_q(\t)=c\t-\nu\bigl(1-\e\bigl(\ee^{-\t e_\eta}\bigr)\bigr)-q=c
\t-\frac{\nu\t
}{\t+\eta}-q.
\]
If either $q>0$ or $c-\nu/\eta\neq0$ (the usual safety loading
condition is $c-\nu/\eta>0$), then the scale function has the form
\[
W_q(x)=u_q\ee^{\Phi_q x}-v_q\ee^{-R_q x},
\]
where $u_q,v_q>0$ and $R_q,\Phi_q\geq0$ (not simultaneously 0).
Moreover, $-R_q$ and $\Phi_q$ are the two roots of $\psi_q(\t)=0$,
see Figure~\ref{fig}
(note that $R_0$ is the classical Lundberg adjustment coefficient),
and
\[
\frac{u_q}{\t-\Phi_q}-\frac{v_q}{\t+R_q}=\frac{1}{\psi_q(\t)}
\]
yielding $u_q=1/\psi_q'(\Phi_q)=\Phi'_q, v_q=-1/\psi_q'(-R_q)=R'_q$.

\begin{figure}[t]

\includegraphics{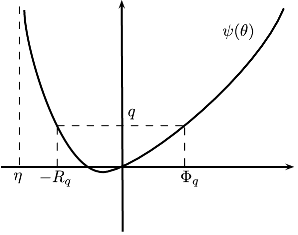}

\caption{The function $\psi(\theta)$ and the inverses $R_q$ and
$\Phi_q$.}\vspace*{-6pt}\label{fig}
\end{figure}

So we also
obtain
\begin{eqnarray*}
Z_q(x,\t)&=&\ee^{\t x} \biggl(1-\psi_q(\t) \biggl(
\frac{u_q}{\Phi_q-\t
}\bigl(\ee^{(\Phi_q-\t) x}-1\bigr)-\frac{v_q}{-R_q-\t}
\bigl(\ee^{(-R_q-\t)
x}-1\bigr) \biggr) \biggr)
\\
&=&\psi_q(\t) \biggl(\frac{u_q\ee^{\Phi_qx}}{\t-\Phi_q}-\frac{v_q
\ee^{-R_q x}}{\t+R_q} \biggr)
\\
&=& \frac{\psi_q(\t)\Phi'_q}{\t-\Phi
_q}\bigl(\ee^{\Phi_qx}-\ee^{-R_qx}\bigr)+\ee^{-R_q x}.
\end{eqnarray*}

Consider~(\ref{eq1}), which for the present model immediately
simplifies to
\begin{eqnarray*}
\e_x \bigl(\ee^{-q T_0^-+\theta X(T_0^-)};T_0^-<\infty
\bigr) &=& \ee^{-R_q
x}\frac{\l-\psi_q(\t) (({\Phi_q-\Phi_{\l+q}})/({\Phi_q-\t}))}{\l
-\psi_q(\t)}
\\[-1pt]
&=& \ee^{-R_q x} \biggl(1-
\frac{\psi_q(\t)}{\psi_{\l+q}(\t
)}\frac{\Phi_{\l+q}-\t}{\Phi_q-\t} \biggr).
\end{eqnarray*}
Since $\psi_{\l+q}(\t)\frac{\t+\eta}{c}=(\t+R_{\l+q})(\t-\Phi
_{\l+q})$, we get
\[
\e_x \bigl(\ee^{-q T_0^-+\theta X(T_0^-)};T_0^-<\infty
\bigr)=\ee^{-R_q
x}\biggl(1-\frac{\t+R_{q}}{\t+R_{\l+q}}\biggr)=\ee^{-R_q x}
\frac{R_{\l
+q}-R_q}{R_{\l+q}+\t},\vadjust{\goodbreak}
\]
which agrees with the result of \cite{actS}, Example~4.2 (take an
exponential penalty function $w_2(y)=\ee^{-\t y}$ for the overshoot).
Note also that $R_{\l+q}\rightarrow\eta$ as $\l\rightarrow\infty
$, because $\t=-\eta$ is the asymptote of $\psi_q(\t)$.
Hence, we also retain the classical formula for the Laplace transform
of the (discounted) ruin deficit under continuous observation
\[
\e_x \bigl(\ee^{-q \tau_0^-+\theta X(\tau_0^-)};\tau_0^-<\infty
\bigr)=\ee^{-R_q x}\frac{\eta-R_q}{\eta+\t},
\]
cf. \cite{GS98}, Equation~(5.42), which can alternatively be obtained
using a
direct argument (exchange the meaning of claims and interarrivals).

Next, identity (\ref{eq2}) simplifies to
\begin{eqnarray*}
&& \e_x \bigl(\ee^{-q T_0^-+\theta X(T_0^-)};T_0^-<\tau_a^+
\bigr)
\\
&&\quad  = \frac
{-\l}{\psi_{\l+q}(\t)}\frac{(\ee^{-R_q a+\Phi_q x}-\ee^{\Phi_q a-R_q
x})({\psi_q(\t)\Phi'_q}/({\t-\Phi_q})-{\l\Phi'_q}/({\Phi
_{\l+q}-\Phi_q}))}{({\l\Phi'_q}/({\Phi_{\l+q}-\Phi_q}))(\ee^{\Phi
_qa}-\ee^{-R_qa})+\ee^{-R_q a}}
\\
&&\quad = \frac{R_{\l+q}-R_q}{R_{\l+q}+\t}\frac{\ee^{\Phi_q
a+R_q(a-x)}-\ee^{\Phi_q x}}{\ee^{\Phi_q a+R_q a}-1+ ({\Phi_{\l
+q}-\Phi_q})/{\l\Phi'_q}}.
\end{eqnarray*}
It is not hard to see that $(\Phi_{\l+q}-\Phi_q)/\l\rightarrow1/c$
as $\lambda\rightarrow\infty$, and hence we have
\begin{eqnarray*}
\e_x \bigl(\ee^{-q \tau_0^-+\theta X(\tau_0^-)};\tau_0^-<\tau
_a^+ \bigr) &=& \frac{\eta-R_q}{\eta+\t}\frac{\ee^{\Phi_q
a+R_q(a-x)}-\ee^{\Phi_q x}}{\ee^{\Phi_q a+R_q a}-1+\psi_q'(\Phi_q)/c}
\\
&=&  \frac{\eta-R_q}{\eta+\t}
\frac{\ee^{\Phi_q a+R_q(a-x)}-\ee^{\Phi_q
x}}{\ee^{\Phi_q a+R_q a}+({R_q-\eta})/({\Phi_q+\eta})},
\end{eqnarray*}
where the last equality follows from the observation that $\frac{\psi
_q(\t)}{\t-\Phi_q}=c\frac{\t+R_q}{\t+\eta}$ and hence $\psi
_q'(\Phi_q)=c\frac{\Phi_q+R_q}{\Phi_q+\eta}$.\vspace*{1pt}

Finally, (\ref{eqdividends1}) provides a formula for the expected
(continuously) discounted dividends until ruin:
%
\begin{eqnarray}\label{diess}
\e_x^a \biggl(\int_0^\infty \ee^{-q t}\1{t<T_0^-}\,\D R(t) \biggr)
&=& \frac{(\lambda\Phi'_q/({\Phi_{\l
+q}-\Phi_q}))(\ee^{\Phi_q
x}-\ee^{-R_q x})+\ee^{-R_q x}}{((\lambda\Phi'_q/(\Phi_{\l+q}-\Phi
_q))(\Phi_q\ee^{\Phi_q a}+R_q\ee^{-R_q a})-R_q\ee^{-R_q a}}\nonumber
\\
&=& \frac{\ee^{\Phi_q x}+\ee^{-R_q x} (({\Phi_{\l+q}-\Phi
_q})/{\lambda\Phi'_q}-1 )}{\Phi_q\ee^{\Phi_q a}-R_q\ee^{-R_q
a} (({\Phi_{\l+q}-\Phi_q})/{\lambda\Phi'_q}-1 )}
\\
&=& \frac{(R_{\l+q}+\Phi
_q)\ee^{\Phi_q x}-(R_{\l+q}-R_q)\ee^{-R_q
x}}{(R_{\l+q}+\Phi_q)\Phi_q\ee^{\Phi_q a}+R_q(R_{\l+q}-R_q)\ee^{-R_q a}},\nonumber
\end{eqnarray}
because $\frac{\Phi_{\l+q}-\Phi_q}{\l\Phi'_q}=\frac{\Phi
_q+R_q}{\Phi_q+R_{\l+q}}$.
This expression is similar to~\cite{dividendsrandomized},
Equation~(24), but
not identical, because there the dividends are paid at Poissonian
times only (see also \cite{avanzi} for a spectrally positive model
setup). Identity (\ref{diess}) is, however, the analogue of \cite{AGS},
Equation~(20), where it was derived for a diffusion process $X(t)$.


\section*{Acknowledgements}
We would like to thank two anonymous referees for helpful remarks
concerning the presentation of this paper.

H. Albrecher and J. Ivanovs are supported by the Swiss National Science
Foundation Project 200020\_143889,
and X. Zhou is supported by an NSERC grant.


%

\printhistory
\end{document}